\documentclass[reqno,11pt]{amsart}
\makeatletter
\let\origsection=\section 
\def\section{\@ifstar{\origsection*}{\mysection}}
\def\mysection{\@startsection{section}{1}\z@{.7\linespacing\@plus\linespacing}{.5\linespacing}{\normalfont\scshape\centering\S}}
\makeatother

\usepackage{amsmath,amssymb,amsthm}
\usepackage{mathrsfs}
\usepackage{dsfont}
\usepackage{mathabx}\changenotsign
\usepackage{xparse}
\usepackage{microtype}

\usepackage{xcolor}
\usepackage[backref=page]{hyperref}
\hypersetup{%
    colorlinks,
    linkcolor={red!60!black},
    citecolor={green!60!black},
    urlcolor={blue!60!black}
}

\usepackage{varioref}

\usepackage[english,algoruled,vlined,resetcount,linesnumbered]{
algorithm2e}

\usepackage{bookmark}

\usepackage[abbrev,msc-links,backrefs]{amsrefs}
\usepackage{doi}

\renewcommand{\PrintDOI}[1]{\doi{#1}}

\usepackage[T1]{fontenc}
\usepackage{lmodern}

\usepackage[english]{babel}

\usepackage{tikz}

\usepackage{geometry}
\usepackage{enumitem}

\let\polishlcross=\l
\def\l{\ifmmode\ell\else\polishlcross\fi}

\renewcommand{\emptyset}{\varnothing}
\renewcommand{\setminus}{\smallsetminus}

\makeatletter
\def\moverlay{\mathpalette\mov@rlay}
\def\mov@rlay#1#2{\leavevmode\vtop{%
   \baselineskip\z@skip \lineskiplimit-\maxdimen
   \ialign{\hfil$\m@th#1##$\hfil\cr#2\crcr}}}
\newcommand{\charfusion}[3][\mathord]{
    #1{\ifx#1\mathop\vphantom{#2}\fi
        \mathpalette\mov@rlay{#2\cr#3}
      }
    \ifx#1\mathop\expandafter\displaylimits\fi}
\makeatother

\newcommand{\sr}{\hat{r}}
\newcommand{\eps}{\varepsilon}

\newtheoremstyle{case}{}{}{\normalfont}{}{\itshape}{:}{ }{}

\newtheorem{thm}[equation]{Theorem}
\newtheorem{lem}[equation]{Lemma}

\newtheorem{conj}[equation]{Conjecture}
\newtheorem{cor}[equation]{Corollary}
\newtheorem{claim}[equation]{Claim}
\newtheorem{ques}[equation]{Question}

\newtheorem{fact}[equation]{Fact}

\theoremstyle{definition}
\newtheorem{defn}[equation]{Definition}

\newtheoremstyle{case}{}{}{\normalfont}{}{\itshape}{\normalfont:}{ }{}

\theoremstyle{case}

\numberwithin{equation}{section}

\newcommand{\cA}{\mathcal{A}}

\def\red{\text{\rm red}}
\def\blue{\text{\rm blue}}

\let\epsilon\varepsilon
\let\:\colon
\let\subset\subseteq

\def\({\left(}
\def\){\right)}
\def\[{\left[}
\def\]{\right]}
\let\<\langle
\let\>\rangle

\def\llceil{\left\lceil}
\def\rrceil{\right\rceil}
\let\ot\leftarrow

\newcommand{\EE}{\mathbb{E}}
\newcommand{\NN}{\mathbb{N}}
\newcommand{\PP}{\mathbb{P}}
\usepackage{datetime}
\usepackage{lineno}
\newcommand*\patchAmsMathEnvironmentForLineno[1]{%
\expandafter\let\csname old#1\expandafter\endcsname\csname #1\endcsname
\expandafter\let\csname oldend#1\expandafter\endcsname\csname end#1\endcsname
\renewenvironment{#1}%
{\linenomath\csname old#1\endcsname}%
{\csname oldend#1\endcsname\endlinenomath}}%
\newcommand*\patchBothAmsMathEnvironmentsForLineno[1]{%
\patchAmsMathEnvironmentForLineno{#1}%
\patchAmsMathEnvironmentForLineno{#1*}}%
\AtBeginDocument{%
\patchBothAmsMathEnvironmentsForLineno{equation}%
\patchBothAmsMathEnvironmentsForLineno{align}%
\patchBothAmsMathEnvironmentsForLineno{flalign}%
\patchBothAmsMathEnvironmentsForLineno{alignat}%
\patchBothAmsMathEnvironmentsForLineno{gather}%
\patchBothAmsMathEnvironmentsForLineno{multline}%
}

\begin{document}

\title{The size-Ramsey number of powers of paths}
\author[
Clemens
\and Jenssen
\and Kohayakawa
\and Morrison
\and Mota
\and Reding
\and Roberts
]
{
Dennis Clemens
\and Matthew Jenssen
\and Yoshiharu Kohayakawa
\and Natasha Morrison
\and Guilherme Oliveira Mota
\and Damian Reding
\and Barnaby Roberts
}

\shortdate
\yyyymmdddate
\settimeformat{ampmtime}
\date{\today, \currenttime}

\address{Technische Universit\"at Hamburg, Institut f\"ur Mathematik, Hamburg, Germany}
\email{\{dennis.clemens|damian.reding\}@tuhh.de}

\address{Department of Mathematics, London School of Economics, London, United Kingdom}
\email{\{m.o.jenssen|b.j.roberts\}@lse.ac.uk}

\address{Instituto de Matem\'atica e Estat\'{\i}stica, Universidade de
	S\~ao Paulo, S\~ao Paulo, Brazil}
\email{yoshi@ime.usp.br}

\address{Mathematical Institute, University of Oxford, Oxford, United Kingdom}
\email{morrison@maths.ox.ac.uk}

\address{Centro de Matem\'atica, Computa\c c\~ao e Cogni\c c\~ao, Universidade Federal do ABC, Santo Andr\'e, Brazil}
\email{g.mota@ufabc.edu.br}

\thanks{%
  The third author was partially supported by FAPESP
  (Proc.~2013/03447-6, 2013/07699-0), by CNPq (Proc.~459335/2014-6,
  310974/2013-5) and by Project MaCLinC/USP.  The fifth author was
  supported by FAPESP (Proc.~2013/11431-2, Proc.~2013/03447-6) and
  partially by CNPq (Proc.~459335/2014-6). The collaboration of part
  of the authors was supported by a CAPES/DAAD PROBRAL grant
  (Proc.~430/15).
}

\begin{abstract}
  Given graphs $G$ and $H$ and a positive integer $q$ say that $G$
  \emph{is $q$-Ramsey for} $H$, denoted
    $G\rightarrow (H)_q$, if every $q$-colouring of the edges of
  $G$ contains a monochromatic copy of $H$. The \emph{size-Ramsey number} $\sr(H)$ of
  a graph $H$ is defined to be
  $\sr(H)=\min\{|E(G)|\colon G\rightarrow (H)_2\}$. Answering a
  question of Conlon, we prove that, for every fixed~$k$, we have
  $\sr(P_n^k)=O(n)$, where~$P_n^k$ is the $k$th power of the
  $n$-vertex path $P_n$ (i.e.~, the graph with vertex set $V(P_n)$ and
  all edges $\{u,v\}$ such that the distance between $u$ and $v$ in
  $P_n$ is at most $k$). Our proof is probabilistic, but can also be made constructive. 
\end{abstract}

\maketitle

\section{Introduction}
\label{sec:intro}

Given graphs $G$ and $H$ and a positive integer $q$ say that $G$
\emph{is $q$-Ramsey for} $H$, denoted $G\rightarrow (H)_q$, if every
$q$-colouring of the edges of $G$ contains a monochromatic copy of
$H$.  When $q = 2$, we simply write $G\rightarrow H$.  In its simplest
form, the classical theorem of Ramsey \cite{Ra} states that for any
$H$ there exists an integer $N$ such that $K_N \rightarrow H$. The
\emph{Ramsey number}~$r(H)$ of a graph~$H$ is defined to be the
smallest such~$N$. Ramsey problems have been well studied and many
beautiful techniques have been developed to estimate Ramsey
numbers. For a detailed summary of developments in Ramsey theory, see the excellent survey of Conlon, Fox and Sudakov~\cite{CoFoSu15}.

A number of variants of the classical Ramsey problem are also under active study. In
particular, Erd\H{o}s, Faudree, Rousseau and Schelp~\cite{ErFaRoSc78}
proposed the problem of determining the smallest number of edges in a
graph $G$ such that $G\rightarrow H$. Define the
\emph{size-Ramsey number} $\sr(H)$ of a graph $H$ to be
$$\sr(H):=\min\{|E(G)|\colon G\rightarrow H\}.$$
In this paper, we are concerned with finding bounds on $\sr(H)$ in some specific cases.

For any graph $H$ it is not difficult to see that
$\sr(H)\leq \binom{r(H)}{2}$.  A result due to Chv\'atal (see,
e.g.,~\cite{ErFaRoSc78}) shows that in fact this bound is tight for complete graphs. For the $n$-vertex path $P_n$, Erd\H{o}s~\cite{Er81} asked the following question.
\begin{ques}
Is it true that
\begin{equation*}
\lim_{n\to\infty}\frac{\sr(P_n)}{n} = \infty \hspace{0.4cm} and \hspace{0.4cm}
\lim_{n\to\infty}\frac{\sr(P_n)}{n^2} = 0?
\end{equation*}
\end{ques}
Using a probabilistic construction, Beck~\cite{Be83} proved that the
size-Ramsey number of paths is linear, i.e., $\sr(P_n)=O(n)$.  Alon
and Chung~\cite{AlCh88} provided an explicit construction of a graph
$G$ with $O(n)$ edges such that $G\rightarrow P_n$.  Recently, Dudek
and Pra{\l}at~\cite{DuPr15} gave a simple alternative proof for this
result (see also~\cite{letzter16:_path_ramsey}).  More generally,
Friedman and Pippenger~\cite{FrPi87} proved that the size-Ramsey
number of bounded degree trees is linear (see
also~\cite{De12,HaKo95,Ke93}) and it is shown in~\cite{HaKoLu95} that
cycles also have linear size-Ramsey numbers.

A question posed by Beck~\cite{Be90} asked
whether $\sr(G)$ is linear for all graphs $G$ with bounded maximum
degree. This was negatively answered by R\"odl and Szemer\'edi, who showed that there exists an $n$ vertex graph $H$
and maximum degree~$3$ such that
$\sr(H)=\Omega(n\log ^{1/60}n)$.  The current best upper bound for 
bounded degree graphs is proved in~\cite{KoRoScSz11}, where it is
shown that for every $\Delta$ there is a constant~$c$ such that for
any graph $H$ with $n$ vertices and maximum degree~$\Delta$:
\begin{equation*}
  \sr(H)\leq cn^{2-1/\Delta}\log^{1/\Delta}n.
\end{equation*}
For further results on size-Ramsey numbers the reader is referred
to~\cite{%
  ben-eliezer12:_Ramsey,
  kohayakawa07+:_ramsey,
  reimer02:_ramsey
}. 

Given an $n$-vertex graph $H$ and an integer $k\geq 2$, the $k$th
power~$H^k$ of $H$ is the graph with vertex set $V(H)$ and all edges
$\{u,v\}$ such that the distance between $u$ and $v$ in $H$ is at most
$k$.  Answering a question of Conlon~\cite{Co} we prove that all
powers of paths have linear size-Ramsey numbers.  The following theorem
is our main result.

\begin{thm}\label{thm:main}
  For any integer $k\geq 2$,
  \begin{equation}
    \label{eq:thm_main}
    \sr(P_n^k)=O(n).
  \end{equation}
\end{thm}

Since $C_n^k\subset P_n^{2k}$, the next corollary follows directly from Theorem~\ref{thm:main}.

\begin{cor}
For any integer $k\geq 2$,
\begin{equation}
  \label{eq:cor_cycle}
  \sr(C_n^k)=O(n).
\end{equation}
\end{cor}

Throughout the paper we use big~$O$ notation with
respect to~$n\to\infty$, where the implicit constants may depend
on other parameters.  For a path $P$, we write~$|P|$ for the
number of vertices in~$P$.  For simplicity, we omit floor
and ceiling signs when they are not essential.

The paper is structured as follows. In Section~\ref{outline} we introduce some preliminary definitions and give an outline of the proof is given. The proof of Theorem \ref{thm:main} is given in
Section~\ref{sec:pf_main_thm}. In Section~\ref{sec:open}, we mention
some related open problems.  

\section{Outline of the proof}\label{outline}
 
To prove Theorem~\ref{thm:main}, we will show that there exists a graph
$G$ with $O(n)$ edges such that~$G\to P_n^k$.  

To construct $G$ we begin by taking a pseudo-random graph $H$ with bounded degree. The existence of such an $H$ will be proved in Lemma~\ref{lem:Hexists}. Given $H^k$, we then take a \emph{complete blow-up}, defined as follows.

\begin{defn}\label{def:blow}
  Given a graph $H$ and a positive integers $t$, the \emph{complete-t-blow-up of} $H$, denoted $H_t$ is the graph obtained by replacing each vertex $v$ of $H$ by a complete graph with $r(K_t)$ vertices, the
  \emph{cluster} $C(v)$, and by adding, for every $\{u,v\}\in E(H)$,
  every edge between $C(u)$ and $C(v)$.
\end{defn}

Note that we replace each vertex with a clique on $r(K_t)$ vertices rather than $t$ vertices as might have been expected.

We will now see that complete blow-ups of powers of
bounded degree graphs have a linear number of edges.  This makes them valid candidates for showing $\sr(P_n^k)=O(n)$.

\begin{fact}\label{fact:linear}
  Let $k$, $t$, $a$ and $b$ be positive constants.  If $H$ is a graph
  with $|V(H)|=an$ and $\Delta (H)\leq b$, then $|E(H^k_t)|=O(n)$.
\end{fact}
\begin{proof}
  Since $\Delta(H) \leq b$, we have $|E(H^k)|=O(n)$. Therefore,
  $|E(H^k_t)|\le r(K_t)^2\cdot|E(H^k)| + r(K_t)^2 \cdot an = O(n)$.
\end{proof}

The heart of the proof is to show that, given any 2-colouring of the edges of $H_t^k$, we can find a monochromatic copy of $P_n$. To do this we will use the fact that $H$ satisfies a particular property (Lemma~\ref{lem:H2}). We shall also make use of the following result.

\begin{thm}[{Pokrovskiy~\cite[{Theorem~1.7}]{Po16}}]
  \label{thm:alexey}
  Let $k \ge 1$. Suppose that the edges of $K_n$ are coloured with red
  and blue.  Then $K_n$ can be covered by $k$ vertex-disjoint blue
  paths and a vertex-disjoint red balanced complete $(k+1)$-partite graph.
\end{thm}

We remark that we do not need the full strength of this result, in the sense
that we do not need the complete $(k+1)$-partite graph to be balanced;
it suffices for us to know that the vertex classes are of comparable
cardinality.  Such a result can be derived easily by iterating
Lemma~1.5 in~\cite{Po16}, for which Pokrovskiy gives a short and
elegant proof (see also~\cite[Lemma~1.10]{pokrovskiy14:_partit}).

We shall also use the classical K\H{o}v\'ari--T. S\'{o}s--Tur\'{a}n
theorem~\cite{KoSoTu54}, in the following simple form.

\begin{thm}
  \label{thm:kst}
  Let~$G$ be a balanced bipartite graph with~$t$ vertices in each
  vertex class.  If~$G$ contains no~$K_{s,s}$, then~$G$ has at
  most~$4t^{2-1/s}$ edges.
\end{thm}

Let us now give a brief outline of how we find our monochromatic copy of $P_n$. Suppose the edges of~$H_t^k$ have been coloured red and blue by a
colouring~$\chi$.  Recall that~$H_t^k$ is obtained by blowing
up~$H^k$; in particular, the vertices~$v$ of~$H^k$ become large
complete graphs~$C(v)$.  By the choice of parameters, Ramsey's theorem
tells us that each such~$C(v)$ contains a monochromatic~$K_t$.  We
suppose that at least half of the~$C(v)$ contain a blue~$K_t$ and
let~$F$ be the subgraph of~$H$ induced by the corresponding
vertices~$v$.

We shall define an auxiliary edge-colouring~$\chi'$ of~$F^k$ and use the fact that~$F^k\to P_n$.  If we
find a blue~$P_n$ in~$F^k$ with the colouring~$\chi'$, then we shall
be able to find a blue~$P_n^k$ in~$H_t^k$.  On the other hand, if no
such blue path~$P_n$ exists in~$F^k$, then we shall be able to find a
\textit{red~$P_n$ in~$F\subset H$} (not in~$F^k$), with certain
additional properties.  More precisely, such a
red~$P_n\subset F\subset H$ will be found as in
Lemma~\ref{lem:H2}, with the sets~$A_i$ being the vertex classes
of a red $(k+1)$-partite subgraph of~$F^k$ as given by
Theorem~\ref{thm:alexey}, applied to a suitable red/blue coloured
complete graph (we complete~$F^k$ with its auxiliary colouring~$\chi'$
to a red/blue coloured complete graph by considering non-edges
of~$F^k$ red).  It will then be easy to find a red~$P_n^k$ in~$H_t^k$.
The idea of defining an auxiliary graph on monochromatic cliques as above was used in \cite{AllBriSko}.

\section{Proof of Theorem~\ref{thm:main}}
\label{sec:pf_main_thm}

Our first lemma guarantees the existence of bounded degree graphs with the pseudo-randomness property we require.

\begin{lem}
  \label{lem:Hexists}
  For every integer $k\geq 1$ and every $\varepsilon>0$ there
  exists~$a_0$ such that the following holds.  For any~$a\geq a_0$
  there is a constant~$b$ such that, for any large enough~$n$, there
  is a graph $H$ with $v(H)=an$ such that:
  \begin{enumerate}
  \item  For every pair of disjoint sets $S$, $T\subset V(H)$ with $|S|,|T|\geq\varepsilon n$, we have $|E_H(S,T)|>0$\label{item-1}.
  \item $\Delta(H)\leq b$. \label{item-2}
  \end{enumerate}
\end{lem}

\begin{proof}
Fix $k\geq 1$ and $\varepsilon>0$.  Let 
\begin{equation}
  \label{eq:a_0_def}
  a_0=2+{4\over\epsilon(k+1)},
\end{equation}
and suppose $a\geq a_0$ is given.  Let
\begin{equation}
  \label{eq:c_def}
  c={4a\over\epsilon^2}
\end{equation}
and
\begin{equation}
  \label{eq:b_def}
  b=4ac.
\end{equation}
Let $n$ be sufficiently large and $G=G(2an,p)$ be the binomial random graph with $p=c/n$.  By
Chernoff's inequality, with high probability we have
$|E(G)|<(4a^2c)n$.  Moreover, with high probability $G$ satisfies~\eqref{item-1} (with $H=G$) by the following reason: Let $X$ be
the number of pairs of disjoint subsets of $V(G)$ of size $\eps n$
with no edges between them.  Then, recalling~\eqref{eq:c_def} and using Markov's inequality, we have
\begin{equation*}
  \PP[X \geq 1] \leq \EE[X]\leq\binom{2an}{\eps n}^2 \left(1-\frac{c}{n}\right)^{(\eps n)^2}
                < 2^{4an}\cdot e^{-c\eps^2 n} = o(1).
\end{equation*}
Thus, we can fix a graph $G$ satisfying these properties.

Now let $H$ be a subgraph of $G$ obtained by iteratively removing a
vertex of maximum degree until exactly $an$ vertices remain.  Then
$\Delta(H)\leq b$, as otherwise we would have deleted more than
$b\cdot an> |E(G)|$ edges from $G$ during the iteration, which, in
view of~\eqref{eq:b_def}, is a contradiction.  Moreover, as $H$ is an
induced subgraph of~$G$, \eqref{item-1} is maintained. This completes the proof of the lemma.
\end{proof}

We now show that any graph satisfying the hypothesis of Lemma~\ref{lem:Hexists} and property \eqref{item-1} also satisfies an additional property. In what follows, $a_0$ will be as defined in Lemma~\ref{lem:Hexists}.

\begin{algorithm}[h]
  \label{alg:cA}
  \caption{}
  \SetKwInOut{Input}{Input}
  \SetKwInOut{Output}{Output}
  \SetKw{STOP}{STOP}

  \Input{a graph~$H$ with~$v(H)=an$ satisfying~\eqref{item-1}
    and sets $A_i\subset V(H)$ ($1\leq i\leq k+1$)
    with~$A_i\cap A_j=\emptyset$ for all~$i\neq j$
    and~$|A_i|\geq\eps an$ for all~$i$.}
  \Output{a path $P_n=(x_1,\dots,x_n)$ in~$H$ with $x_i \in A_j$ for
    all~$i$, where~$j\equiv i\pmod{k+1}$.} 

  \ForEach {$1\leq i\leq k+1$} {
    $U_i\ot A_i$;\quad $D_i\ot\emptyset$
  }
  \While {$|D_i|\leq|A_i|/2$ for all~$i$} {
    pick~$x_1\in U_1$ and let~$P=(x_1)$;\quad $r\ot 1$;\quad $U_1\ot U_1\setminus\{x_1\}$\label{alg:U1_empty} \\
    \While {$1\leq|P|<n$} {
      \tcp{$P=(x_1,\dots,x_r)$ with~$r\geq1$}
      \eIf {$\exists u\in U_{r+1}$ with $\{x_r,u\}\in E(H)$} {
        \label{alg:if}
        $x_{r+1}\ot u$;\quad $U_{r+1}\ot U_{r+1}\setminus\{u\}$\\
        $P\ot(x_1,\dots,x_r,x_{r+1})$;\quad $r\ot r+1$
      } {
        $D_r\ot D_r\cup\{x_r\}$\label{alg:declare_d}\\
        $P\ot(x_1,\dots,x_{r-1})$;\quad $r\ot r-1$
      }
    }
    \If {$|P|=n$} {
      \Return $P$\quad\tcp{path has been found}
      \label{alg:returns_P}
    }
  }
  \STOP\ with failure\quad\tcp{this will not happen}  
  \label{alg:deadend}
\end{algorithm}

\begin{lem}\label{lem:H2}
Let $H$ be a graph with $v(H) = an$, for $a \ge a_0$, with property \eqref{item-1}. Then, for any family of pairwise disjoint sets
    $A_1,\dots,A_{k+1}\subseteq V(H)$ each of size at least $\eps an$, there is a path $P_n=(x_1,\dots,x_n)$ in~$H$ with
    $x_i \in A_j$ for all~$i$, where $j\equiv i\pmod{k+1}$. 
\end{lem}

To prove this lemma, we analyse a depth first search algorithm, adapting a proof
idea in~\cite[Lemma~4.4]{ben-eliezer12:_Ramsey}.  More specifically,
we run an algorithm (stated formally as Algorithm~\ref{alg:cA}).  Our algorithm
receives as input a graph~$H$ with $v(H)=an$ satisfying
Property \eqref{item-1} and a family of pairwise disjoint sets
$A_1,\dots,A_{k+1}\subseteq V(H)$ with~$|A_i|\geq\eps an$ for all~$i$.
The output of~$\cA$ is a path $P_n=(x_1,\dots,x_n)$ in~$H$ with
$x_i\in A_j$ for all~$i$, where $j\equiv i\pmod{k+1}$.

As it runs, the algorithm builds a path~$P=(x_1,\dots,x_r)$ with
$x_i\in A_j$ for all~$i$ and~$j$ with $j\equiv i\pmod{k+1}$.
Furthermore, it maintains sets~$U_j$ and~$D_j\subset A_j$ for all~$j$,
with the property that $U_j$, $D_j$, and~$V(P)\cap A_j$ form a
partition of~$A_j$ for every~$j$.  The cardinality of the sets~$U_j$ decrease as the
algorithm runs, while the~$D_j$ increase.  As the algorithm runs, we
have~$r=|P|<n$ and it searches for an edge~$\{x_r,u\}\in E(H)$
where~$u$ belongs to the set~$U_{r+1}$ of \textit{unused} vertices
in~$A_{r+1}$.  If such a vertex~$u\in U_{r+1}$ is found, then~$P$ is
made one vertex longer by adding~$u$ to it.  If there is no such
vertex~$u$, then~$x_r$ is declared a \textit{dead end} and it is put
into~$D_r$.  Moreover, the path~$P$ is shortened by one vertex; it
becomes $P=(x_1,\dots,x_{r-1})$.  Our algorithm iterates this
procedure.  If we find a path~$P$ with~$n$
vertices this way, then we are done. 

We now analyse Algorithm~\ref{alg:cA}.

\begin{proof}[Proof of Lemma~\ref{lem:H2}]

We will prove that Algorithm~\ref{alg:cA} returns a
path~$P$ on line~\ref{alg:returns_P} as desired, instead of
terminating with failure on line~\ref{alg:deadend}.  

First recall that $U_i$, $D_i$, and~$V(P)\cap A_i$ form a
partition of~$A_i$ for every~$i$.
Since the path $P$ is always empty on line~\ref{alg:U1_empty}, at this point we have 
$|U_1|\geq|A_1|-|D_1|\geq|A_1|/2>0$. Then, line~\ref{alg:U1_empty} is always executed succesfully.

Suppose now that~$\cA$ stops with failure on line~\ref{alg:deadend}.
Then, for some~$i$, say~$i=r$, the set~$D_i=D_r$ became larger
than~$|A_r|/2\geq\epsilon an/2\geq\epsilon n$.  Furthermore, we
have~$|P|<n$ and~$|D_{r+1}|\leq|A_{r+1}|/2$ (indices modulo~$k+1$) and
hence, 
\begin{equation*}
	|U_{r+1}|\geq|A_{r+1}|-|D_{r+1}|-|V(P)\cap A_{r+1}|
	\geq{1\over2}|A_{r+1}|-\llceil n\over k+1\rrceil
	\geq{1\over2}\epsilon an-{2n\over k+1}
	>\epsilon n.
\end{equation*}
Applying Property \eqref{item-1} of Lemma~\ref{lem:Hexists} to the
pair~$(D_r,U_{r+1})$, we see that there is an edge~$\{x,u\}\in E(H)$
with~$x\in D_r$ and~$u\in U_{r+1}$.  Consider the moment in which~$x$
was put into~$D_r$.  This happened on line~\ref{alg:declare_d},
when~$P$ had~$x$ as its foremost vertex and~$\cA$ was trying to
extend~$P$ further into~$U_{r+1}$.  At this point, because of the
edge~$\{x,u\}\in E(H)$, we must have had~$u\notin U_{r+1}$ (see
line~\ref{alg:if}).  Since the set~$U_{r+1}$ decreases as~$\cA$ runs,
this is a contradiction and hence~$\cA$ does not terminate on
line~\ref{alg:deadend}.

Algorithm~\ref{alg:cA} terminates
as~$\sum_{1\leq i\leq k+1}\big(|D_i|-|U_i|\big)$ increases as it runs.
We conclude that it returns a suitable path~$P$ as claimed.
\end{proof}

We are now ready to complete the proof of Theorem~\ref{thm:main}.

\begin{proof}[Proof of Theorem~\ref{thm:main}]
  Fix $k\geq 1$ and let~$\eps=1/3(k+1)$.  Let~$a_0$ be the constant
  given by an application of Lemma~\ref{lem:Hexists} with
  parameters~$k$ and~$\eps$.  Set $a=\max\{6k,a_0\}$ and let~$b$ be
  given by Lemma~\ref{lem:Hexists} for this choice of~$a$.  Moreover,
  let~$H$ be a graph with $|V(H)|=an$ and $\Delta(H)\leq b$ be as in
  Lemma~\ref{lem:Hexists}.  Finally, put $t=(64k)^{2k}$ and $s=2k$.

  Let $H^k_t$ be a complete-t-blow-up of $H^k$, as in
  Definition~\ref{def:blow}, and let
  $\chi\colon E(H^k_t)\to\{\red,\blue\}$ be an edge-colouring
  of~$H^k_t$.  We shall show that $H^k_t$ contains a monochromatic
  copy of~$P_n^k$ under~$\chi$.  By the definition of $H^k_t$, any
  cluster $C(v)$ contains a monochromatic copy $B(v)$ of $K_t$.
  Without loss of generality, the set
  $W:=\{v \in V(H)\colon B(v) \text{ is blue}\}$ has cardinality at
  least~$v(H)/2$.  Let $F:= H[W]$ be the subgraph of $H$ induced by
  $W$, and let $F'$ be the subgraph of $F_t^k\subset H^k_t$ induced by
  $\bigcup_{w \in W}V(B(w))$.

  Given the above colouring $\chi$, we define a colouring $\chi'$ of
  $F^k$ as follows.  An edge $\{u,v\}\in E(F^k)$ is coloured
  \emph{blue} if the bipartite subgraph~$F'[V(B(u)),V(B(v))]$
  of~$F'$ naturally induced by the sets~$V(B(u))$ and~$V(B(v))$
  contains a blue~$K_{s,s}$.  Otherwise $\{u,v\}$ is coloured
  \emph{red}.

  \begin{claim}\label{lem:auxiliary}
    Any $2$-colouring of $E(F^k)$ has either a blue $P_n$ or a red~$P_n^k$.
  \end{claim}

  \begin{proof}
    We apply Theorem~\ref{thm:alexey} to $F^k$, where if an edge is
    not present in~$F^k$, then we consider it to be in the red colour
    class.  If $F^k$ contains a blue copy of $P_n$, then we are done.
    Hence we may assume $F^k$ contains a balanced, complete
    $(k+1)$-partite graph $K$ with parts $A_1,\dots,A_{k+1}$ on at
    least $v(F^k)-kn\geq an/2-kn$ vertices, with no blue edges between
    any two parts.  As~$a\geq6k$, each one of these parts has size at
    least
    \begin{equation}
      \label{eq:A_i_sizes}
      {1\over k+1}\({1\over2}a-k\)n\geq\epsilon an.
    \end{equation}
    By Property~\ref{lem:H2} of Lemma~\ref{lem:Hexists} applied to the
    collection of sets of vertices $A_1,\dots,A_{k+1}$ of $F\subset H$
    (specifically $F$ and not~$F^k$), we see that $F[V(K)]$ contains a
    path with $n$ vertices such that any consecutive $k+1$ vertices
    are in distinct parts of~$K$.  Therefore $F^k[V(K)]$ contains a
    copy of $P_n^k$ in which every pair of adjacent vertices are in
    distinct parts of~$K$. By definition of~$K$, such a copy is red.
  \end{proof}

  By Claim~\ref{lem:auxiliary}, $F^k$ contains a blue copy of $P_n$ or
  a red copy of $P_n^k$ under the edge-colouring $\chi'$.  Thus, we can
  split our proof into these two cases.

  (\textit{Case~1})\quad First suppose $F^k$ contains a blue copy
  $(x_1,\dots,x_n)$ of~$P_n$.
  Then, for every $1 \le i \le n-1$, the bipartite graph
  $F'[V(B(x_i)),V(B(x_{i+1}))]$ contains a blue copy of $K_{s,s}$, with,
  say, vertex classes $X_i \subseteq V(B(x_i))$ and
  $Y_{i+1}\subseteq V(B(x_{i+1}))$.  As $|X_i|=|Y_i|=s=2k$ for all
  $2\leq i\leq n-1$, we can find sets $X'_i \subseteq X_i$ and
  $Y'_i\subseteq Y_i$ such that $|X'_i| = |Y'_i| = k$ and
  $X'_i \cap Y_i' = \emptyset$ for all $2 \le j \le n-1$.
  Let~$X_1'=X_1$ and~$Y_n'=Y_n$.

  We now show that the set
  $U := \bigcup_{i=1}^{n-1}X'_i \cup \bigcup_{i=2}^{n}Y'_i$ provides
  us with a blue copy of~$P_{2kn}^k$ in~$F'\subset H_t^k$.  Note first
  that~$|U|=2k+2k(n-2)+2k=2kn$.  Let $u_1,\dots, u_{2kn}$ be an
  ordering of $U$ such that, for each~$i$, every vertex in $X_i'$
  comes before any vertex in $Y_{i+1}'$ and after every vertex in
  $Y_i'$.  By the definition of the sets $X_i'$ and $Y_i'$ and the
  construction of $F'\subset F_t^k\subset H_t^k$, each vertex~$u_j$ is
  adjacent in blue to $\{u_{j'}\in U\: 1 \le |j-j'| \le k \}$.  Thus,
  $U$~contains a blue copy of~$P_{2nk}^k$, as claimed.

  (\textit{Case~2})\quad Now suppose $F^k$ contains a red copy~$P$
  of~$P^k_n$. That is, $F^k$ contains a set of vertices
  $\{x_1,\dots,x_n\}$ such that $x_i$ is adjacent in red to all $x_j$
  with $1 \le |j-i| \le k$.  We shall show that, for each
  $1 \le i \le n$, we can pick a vertex $y_i\in V(B(x_i))$ so that
  $y_1,\dots, y_n$ define a red copy of~$P_n^k$
  in~$F'\subset F_t^k\subset H_t^k$.  We do this by
  applying the local lemma~\cite[p.~616]{erdos75:_probl} (a greedy strategy also works).

  We have to show that it is possible to pick the~$y_i$
  ($1\leq i\leq n$) in such a way that~$\{y_i,y_j\}$ is a red edge
  in~$F'$ for every~$i$ and~$j$ with $1 \le |i-j| \le k$.  Let us
  choose~$y_i\in V(B(x_i))$ ($1\leq i\leq n$) uniformly and
  independently at random.  Let~$e=\{x_i,x_j\}$ be an edge
  in~$P\subset F^k$.  We know that~$e$ is red.  Let~$A_e$ be the event
  that~$\{y_i,y_j\}$ is a \textit{blue} edge in~$F'$.  Since the
  edge~$e$ is red, we know that the bipartite
  graph~$F'[V(B(x_i)),V(B(x_j))]$ contains no blue~$K_{s,s}$.
  Theorem~\ref{thm:kst} then tells us that~$\PP[A_e]\leq4t^{-1/s}$.

  The events~$A_e$ are not independent, but we can define a dependency
  graph~$D$ for the collection of events~$A_e$ ($e\in E(P)$) by adding
  an edge between~$A_e$ and~$A_f$ if and only
  if~$e\cap f\neq\emptyset$.  Then~$\Delta(D)\leq4k$.  Given that 
  \begin{equation}
    \label{eq:4epd}
    4\Delta\PP[A_e]\leq64kt^{-1/s}=1
  \end{equation}
  for all~$e$, the Local Lemma tells us
  that~$\PP\big[\bigcap_{e\in E(P)}\bar A_e\big]>0$, and hence a
  simultaneous choice of the~$y_i$ ($1\leq i\leq n$) as required is
  possible.  This completes the proof of Theorem~\ref{thm:main}.
  \end{proof}

Throughout our proof we have used probabilistic methods to show the existence of $G$. We now briefly discuss how our proof could be made constructive.
First observe that one could give an explicit construction of~$H$. For instance, it suffices to take
for~$H$ a suitable $(n,d,\lambda)$-graph as in Alon and
Chung~\cite{AlCh88}, namely, it is enough to have~$\lambda=O(\sqrt d)$
and~$d$ large enough with respect to~$k$ and~$1/\epsilon$.

\section{Open questions}
\label{sec:open}

We make no attempts to optimise the constant given by our argument, so the following question is of interest.

\begin{ques}
For any integer $k\geq 2$, what is $\limsup\limits_{n\longrightarrow\infty}(\sr(P_n^k)/n)$?
\end{ques}

It is also interesting to consider what happens when more than two colours are at play. For $q \in \NN$, let $\sr_q(H)$ denote the \emph{$q$-colour size-Ramsey number} of $H$; the smallest number of edges in a graph that is $q$-Ramsey for $H$.

\begin{conj}
For any $q,k \in \NN$ we have $\sr_q(P_n^k)=O(n)$.
\end{conj}

It is conceivable that in hypergraphs the size-Ramsey number (defined analogously as for graphs) of tight paths may be linear.
Let $H^{(k)}_n$ denote the tight path of uniformity $k$ on $n$ vertices; that is $V(H^{(k)}_n)=[n]$ and $E(H^{(k)}_n)=\big\{\{1,...,k\},\{2,...,k+1\},...,\{n-k+1,...,n\}\big\}$. The following question appears as Question 2.9 in \cite{DuFlMuRo16+}.

\begin{ques}
For any $k \in \NN$, do we have $\sr(H^{(k)}_n)=O(n)$?
\end{ques}

Finally we note that for fixed $k$, our main result implies the linearity of the size Ramsey number for the grid graphs $G_{k, n}$, the cartesian product of the paths $P_k$ and $P_n$. Indeed our main result implies the linearity of the size Ramsey number for any sequence of graphs with bounded bandwidth. For the $d$-dimensional grid graph $G_n^d$, obtained by taking the cartesian product of $d$ copies of $P_n$, we raise the following question.

\begin{ques}
For any integer $d\geq 2$, is $\sr(G_n^d)=O(n^d)$?
\end{ques}

\section{Acknowledgements}

This research was conducted while the authors were attending the
ATI--HIMR Focused Research Workshop: \textit{Large-scale structures in
  random graphs} at the Alan Turing Institute. 
We would like
to thank the organisers of this workshop and the institute for their
facilitation of productive research environment and provision of a
magical coffee machine.
We are very grateful to David Conlon for suggesting the problem~\cite{Co}.

\bibliographystyle{amsplain}
\bibliography{ramsey}

\end{document}